\documentclass[11pt,a4paper]{article}
\usepackage{amsmath,amsfonts,amsxtra,amssymb,amscd,amsthm}
\usepackage{indentfirst,anysize,graphicx}
\marginsize{3cm}{2cm}{1cm}{1cm}
\newtheorem{theorem}{Theorem}[section]
\newtheorem{corollary}{Corollary}[section]
\newtheorem{proposition}{Proposition}[section]
\newtheorem{lemma}{Lemma}[section]
\newtheorem{definition}{Definition}[section]
\newtheorem{remark}{Remark}[section]

\begin{document}

\title{Classical and non-classical central limit theorems for random sums of independent random variables of a double sequence}
\author{Tran Loc Hung\footnote{Ho Chi Minh City, Vietnam. Email: tlhungvn@gmail.com}}
\maketitle

Dedicated to Academician Sh. K. Formanov (Academy of Sciences of Uzbekistan) on the occasion of his 84th birthday

\begin{abstract}
Since the appearance of H. Robbins article (1948), the central limit theorems for random sums have been studied for about 70 years. The central limit theorems for random sums of independent random variables play a very important role in various disciplines such as statistics, financial mathematics, insurance, etc. The purpose of this paper is to randomize some well-known classical and non-classical central limit theorems for sums of independent (not necessarily identically distributed) random variables of a double sequence, with the conditions for determining their validity. The results obtained in this paper are extensions and generalizations of known ones.
\end{abstract}


\section{Introduction}\label{sec:1}
Let $\{X_{n, j}\}:=\{X_{n, j}: 1\leq j\leq k_{n}, n\geq 1\}$  be a double sequence of independent in each row random variables, which satisfies the following conditions for classical triangular arrays:
\begin{enumerate}
\item[{\bf (A)}.] For each $n\geq 1,$ the random variables $X_{n, 1}, X_{n, 2}, \ldots, X_{k_{n}}, \ldots$ in the nth row are mutually independent (not necessary identically distributed);
\item[{\bf (B)}.] The double sequence $\{X_{n, j}\}$ has the zero means $\mathbb{E}X_{n, j}=0$ for $n\geq 1$ and $1\leq j\leq k_{n};$
\item[{\bf (C)}.] The sum of all variances equals one, i.e., $\sum\limits_{j=1}^{k_{n}}\mathbb{D}X_{n, j}=\sum\limits_{j=1}^{k_{n}}\sigma^{2}_{n, j}=1.$
\end{enumerate}
From now on, the sequence of natural numbers $\{k_{n}>0\}$ is suggested that $k_{n}\to\infty$ as $n\to\infty.$ For any $n\geq 1,$ set 
\begin{equation}\label{equ:1.1}
S_{n, k_{n}}:=X_{n, 1}+X_{n, 2}+\dots+ X_{n, k_{n}}
\end{equation}
the partial sums of $X_{n, j},$ and $B^{2}_{n, k_{n}}:=\mathbb{D}S_{n, k_{n}}$ denotes the variance of the sums $S_{n, k_{n}}.$\\
From the conditions ({\bf B}) and ({\bf C}), it is clear that
 \[
 \mathbb{E} S_{n, k_{n}}=0 \quad \text{and}\quad \mathbb{D} S_{n, k_{n}}=1.
 \]
Now, when the independent summands increase to infinity, we are interested in asymptotic behavior of the probability distributions of the following normalized sums: 
\begin{equation}\label{equ:1.2}
\frac{S_{n, k_{n}}-\mathbb{E}S_{n, k_{n}}}{\sqrt{\mathbb{D}S_{n, k_{n}}}}=\sum\limits_{j=1}^{k_{n}}X_{n, j}.
\end{equation}
Define
\begin{equation}\label{equ:1.3}
\Delta_{n, k_{n}}=\sup\limits_{x\in\mathbb{R}}\bigg|P\bigg(\sum\limits_{j=1}^{k_{n}}X_{n, j}<x\bigg)-\Phi_{0,1}(x) \bigg|,
\end{equation}
where $\Phi_{0,1}(x)=\frac{1}{\sqrt{2\pi}}\int\limits_{-\infty}^{x}e^{-y^{2}/2}$ stands for the distribution function of the standard normally distributed random variable, denoted by $X^{*}\stackrel{D}{\sim}\mathcal{N}(0,1).$ 
\begin{definition}\label{def:1.1}\quad A double sequence $\{X_{n, j}\}$ obeys the central limit theorem if 
\begin{equation*}
\Delta_{n, k_{n}}=o(1)\quad\text{as}\quad n\to\infty.\tag{CLT}
\end{equation*}
\end{definition}
The central limit theorem (CLT) shows that the distribution of a large number of independent (in each row) small random variables $X_{n, j}$ for $n\geq 1, 1\leq j\leq k_{n},$ can be approximated by a standard normal distribution when the number of independent individual summands is large enough. The notion of  "small random variables" may be described by the following concept.
\begin{definition}\label{def:1.2}\quad The random variables $X_{n, j}$ for $n\geq 1, 1\leq j\leq k_{n},$ are called to be asymptotic infinitesimal if
\begin{equation*}
\max\limits_{1\leq j\leq k_{n}}P(|X_{n, j}|\geq \epsilon)=o(1)\quad\text{as}\quad n\to\infty,\tag{I}
\end{equation*}
for any $\epsilon >0.$ 
\end{definition}
Thus, for $n\geq1,$ the normalized partial sums $\sum\limits_{j=1}^{k_{n}}X_{n, j}$ are formed by the "small" independent individual summands $X_{n, j}$ and their limiting distributions are non-degenerate since the number of independent individual summands is large enough. A family of the central limit theorems is divided into two classes depending on the use of the condition (I). Following Zolotarev \cite{Zolotarev1997}, theorems on limit distributions for the normalized summations $\sum\limits_{j=1}^{k_{n}}X_{n, j}$ proved without the condition (I) are said to be non-classical. The central limit theorems in classical versions are the Lindeberg-Feller's, and Lyapunov's  theorems. Besides, the non-classical central limit theorems were first studied by P. L{\'e}vy, and an extension of the Lindeberg-Feller central limit theorem for the sum of independent random variables in a triangular array in a non-classical situation was due to Zorotarev (1967) and Rotar (1975) (see \cite{Rotar1975} for more details). For a deeper discussion related to the non-classical Rotar's theorem, we refer the reader to \cite{Formanov2019},  \cite{Formanov2022}, \cite{Ibragimov2020}, \cite{Rotar1975}, \cite{Rotar1996}, and \cite{Shiryaev1996}. 

Since the appearance of Robbins' article \cite{Robbins1948}, the limit theorems for random sums have been studied for about 70 years now.  The most of the limit results concerning random sums of independent random variables, such as the random-sum central limit theorem, and their importance in various disciplines such as financial mathematics and insurance, are studied. The central limit theorems for random sums are mentioned in many classical documents such as \cite{Blum1963}, \cite{Butzer1983}, \cite{Chen2011}, \cite{Cioczek1987},\cite{Feller1971}, \cite{Gnedenko1996}, \cite{Gut2005}, \cite{Kalashnikov1997}, \cite{Kruglov1990}, \cite{Neammanee2004},  \cite{Renyi1970},  \cite{Robbins1948}, \cite{Rychlik1976}, \cite{Rychlik1979}, \cite{Shanthikumar1984}, etc.

The main motivation of this study is to randomize the  classical and non-classical central limit theorems of independent in each row (not necessary identically distributed) random variables of the double sequences, when the non-random index $k_{n}$ of the partial sums $S_{n, k_{n}}$ is replaced by a sequence $(\nu_{n}, n\geq 1)$ of  integer-valued positive random variable,  independent of all independent individual summands $X_{n, j}$ for $n\geq 1, 1\leq j\leq k_{n}.$ From now on,  we assume that $\nu_{n}\stackrel{P}{\longrightarrow}\infty$ when $n\to\infty.$  For $n\geq 1,$ the sums of a random number $\nu_{n}$ of all $X_{n, j},$ defined  in form
 \begin{equation}\label{equ:1.4}
 S_{n, \nu_{n}}:=X_{n, 1}+X_{n, 2}+\ldots+X_{n, \nu_{n}},
 \end{equation}
are so-called random sums. The random sums are named like Poisson, geometric, or negative binomial random sums, and so on,  depending on the distributions of random variables $\nu_{n}, n\geq 1.$ (see  \cite{Feller1971}, \cite{Gnedenko1996}, \cite{Gut2005},  \cite{Kalashnikov1997}, \cite{Kruglov1990},   \cite{Renyi1970}, \cite{Robbins1948},  and \cite{Hung2021}, for more details). 

 For each $n\geq 1,$ set
\[
\Delta_{n, \nu_{n}}=\sup\limits_{x\in\mathbb{R}}\bigg|P\bigg(\sum\limits_{j=1}^{\nu_{n}}X_{n, j}<x\bigg)-\Phi_{0, 1}(x) \bigg|.
\]
The central limit theorems for random sums are an extension of the central limit theorem (CLT) in following form:
\begin{definition}\label{def:1.3}\quad A double sequence $\{X_{n, j}\}$ obeys the random central limit theorem if 
\begin{equation*}
\Delta_{n, \nu_{n}}=o(1)\quad\text{as}\quad n\to\infty,\tag{RCLT}.
\end{equation*}
\end{definition}
Following  \cite{Gut2005}, we have
\[
\Delta_{n, \nu_{n}}=\sum\limits_{k_{n}=1}^{\infty}P(\nu_{n}=k_{n})\Delta_{n, k_{n}}=
\sum\limits_{k_{n}=1}^{\infty}P(\nu_{n}=k_{n})\sup\limits_{x\in \mathbb{R}}\bigg|P\bigg(\sum\limits_{j=1}^{k_{n}}X_{n, j}<x\bigg)-\Phi_{0,1}(x) \bigg|.
\]
Thus,  if $\nu_{n}\stackrel{a.s.}{=}k_{n}$ for $n\geq 1,$ then the central limit theorem for random sum (RCLT) deduces the non-random central limit theorem (CLT). 

A central limit theorem for random sum of  independent and identically distributed (i.i.d.) random variables  was first studied by Robbins \cite{Robbins1948}. The case of independent (not necessary identically distributed ) random variables was also considered later in \cite{Chen2005}, \cite{Cioczek1987}, \cite{Feller1971}, \cite{Gnedenko1996}, \cite{Gut2005},  \cite{Kruglov1990}, \cite{Neammanee2004}, \cite{Renyi1970}, \cite{Rychlik1976}, and \cite{Rychlik1979}, etc. 

 The main aim of this paper is to randomize some well-known classical and non-classical central limit theorems for sums of independent (not necessarily identically distributed) random variables of a double sequence, with the conditions for their validity. In this article, the central limit theorems for random sums representing classical classes such as Lindeberg's, Lyapunov's, and Feller's  are investigated. The Rotar's central limit theorem for random sums is also introduced in this paper as a result in non-classical situations.  The mathematical tools used  for study in this paper are characteristic functions, Chebyshev's inequality,  and Zolotarev distance

The layout of this paper is as follows: The second section provides some preliminary information on classical and non-classical central limit theorems, providing familiar conditions such as the Lindeberg's, Lyapunov's, Feller's, and Rotar's conditions. The well-known central limit theorems in classical and non-classical situations, like the Lindeberg's, Lindeberg-Feller's,  and Rotar's, are introduced. In Section 3, some definitions related to the random conditions in Lindeberg's, Lyapunov's,  Feller's, and Rotar's  types and their relationships are presented. The random condition of asymptotic  infinitesimal for the independent individual summands $X_{n, j}$ is defined, and the equivalent statements are established.  Section 4 is devoted to randomizing the central limit theorems for random sums of independent (not necessarily identically distributed) random variables of a double sequence like Lindeberg's, Lyapunov's, Lindeberg-Feller's, and Rotar's theorems. The last section closes with concluding remarks on the  relationships between the  conditions and the  central limit theorems for random sums of independent random variables of a  sequence $(X_{j}, j\geq 1)$ in classical and non-classical situations when the "series form" is used.

In this paper, some notations are used, like the $X^{*}\stackrel{D}{\sim}\mathcal{N}(0,1)$ means that the standard normally distributed random variable $X^{*}$ follows the standard normal distribution function $\Phi_{0,1}(x),$ and the $\stackrel{P}{\longrightarrow}\infty$  denotes the convergence in probability. The symbols $X\stackrel{a.s.}{=}Y$ and $X\stackrel{D}{=}Y$ show that the equality is almost surely and  the equality in distribution, respectively.
The expression $a_{n}=o(1)$ as $n\to\infty$ stands for $\lim\limits_{n\to\infty}a_{n}=0.$

\section{Preliminary information}\label{sec:2}

\subsection{Conditions for the validity of the central limit theorems}
The information in this part concerning the conditions in determining for validity of the non-random central limit theorems (CLT). The references \cite{Feller1971},  \cite{Gnedenko1949}, \cite{Petrov1975}, \cite{Petrov1995}, \cite{Renyi1970}, \cite{Rotar1996},  \cite{Shiryaev1996}, and \cite{Zolotarev1997} are devoted to the study of the central limit theorems for sums of independent random variables and the  conditions for the validity of those theorems. 

We assume that, for $n\geq 1$ and $1\leq j\leq k_{n},$ the desired random variables  $X_{n, j}$ are satisfied the condition of asymptotic infinitesimality (I),   Recalling a result in \cite{Petrov1995} (see  Lemma 3.1 on page 89).
\begin{proposition}\label{pro:2.1}\quad The condition (I) is equivalent to the following condition:    
\begin{enumerate}
\item[{\bf A}.] \quad As $n\to\infty,$
\[
\max\limits_{1\leq j\leq k_{n}}\mathbb{E}\bigg(\frac{X^{2}_{n, j}}{1+X^{2}_{n, j}} \bigg)=o(1),
\]
\item[{\bf B}.] \quad For any $t\in\mathbb{R},$ 
\[
\max\limits_{1\leq j\leq k_{n}}\bigg|f_{n, j}(t)-1\bigg|=o(1), \quad\text{as}\quad n\to\infty,
\]
here  $f_{n, j}(t):=\mathbb{E} e^{itX_{n, j}}$ stands for the characteristic functions of random variables $X_{n, j}.$
\end{enumerate}
\end{proposition}
It is worth pointing out that the constraint (I) is needed if we want to make the limiting law for the distributions of the normalized sums $S_{n, k_{n}}$ insensitive to the behavior of independent individual summands $X_{n, j}$ for $n\geq 1$ and $1\leq j\leq k_{n}.$ The construction of the theory of the summation of independent random variables when condition (I) is met is mainly associated with the names of P. L{\'e}vy, W. Feller, A.Ya. Khinchin, A.N. Kolmogorov, and B.V. Gnedenko (see \cite{Gnedenko1949} for more details). In particular, these mathematicians have proved that the class of limit distribution laws for $F_{n, k_{n}}(x):=P\bigg(\sum\limits_{j=1}^{k_{n}}X_{n, j}<x\bigg)$ (in the sense of weak convergence) coincides with the class of all infinitely divisible distributions (see  \cite{Gnedenko1949} and \cite{Petrov1995}, for more details). 

The next conditions are those for the determining of the central limit theorems.
\begin{definition}\label{def:2.1}\quad A double sequence $\{X_{n, j}\}$ is said  to satisfy the Feller condition if  
\begin{equation*}
\max\limits_{1\leq j\leq k_{n}}\sigma^{2}_{n, j}=o(1). \quad\text{as}\quad n\to\infty\tag{F}
\end{equation*}
\end{definition}
\begin{proposition}\label{pro:2.2}\quad The condition (I) is deduced by the Feller condition (F), i.e., 
\[
(F)\Longrightarrow (I) \quad\text{as}\quad n\to\infty.
\]
\end{proposition}
\begin{proof}
By Chebyshev's inequality, for any $\epsilon>0,$ we have
\[
\max\limits_{1\leq j\leq k_{n}}P(|X_{n, j}|\geq \epsilon)\leq \frac{1}{\epsilon^{2}}\max\limits_{1\leq j\leq k_{n}}\sigma^{2}_{n, j}.
\]
It is clear that, from the above inequality,  the condition (I) is satisfied if the condition (F) holds.
\end{proof}
 The existence of the condition (F) also explains why the number $k_{n}$ of independent individual summands $X_{n, j}$  in the normalized sums $\sum\limits_{j=1}^{k_{n}}X_{n, j}$ must be large enough. 
 \begin{proposition}\label{pro:2.3}\quad The $k_{n}\longrightarrow\infty$ as $n\to\infty,$ f the Feller condition (F) holds.
\end{proposition}
\begin{proof}
By the expression $ \sum\limits_{j=1}^{k_{n}}\sigma^{2}_{n, j}=1,$ we conclude that
\[
1=\sum\limits_{j=1}^{k_{n}}\sigma^{2}_{n, j}\leq k_{n}\times \bigg(\max\limits_{1\leq j\leq k_{n}}\sigma^{2}_{n, j} \bigg),
\]
and it follows that
\[
k_{n}\geq \bigg(\frac{1}{\max\limits_{1\leq j\leq k_{n}}\sigma^{2}_{n, j}} \bigg)\longrightarrow +\infty\quad\text{as}\quad n\to\infty,
\]
if the Feller condition (F) holds.
\end{proof}
\begin{definition}\label{def:2.2}\quad A double sequence $\{X_{n, j}\}$ is said  to satisfy the Lindeberg condition if for any $\epsilon >0$ 
\begin{equation*}
\sum\limits_{j=1}^{k_{n}}\int\limits_{|x|>\epsilon}x^{2}dF_{n, j}(x)=o(1)\quad\text{as}\quad n\to\infty,\tag{L}
\end{equation*}
where $F_{n, j}(x):=P(X_{n, j}<x)$ denotes the distribution function of a random summand $X_{n, j}$ for $n\geq 1$ and $1\leq j\leq k_{n}.$
\end{definition}
\begin{proposition}\label{pro:2.4}\quad For any $\epsilon>0,$ the Feller condition (F) is followed from the Lindeberg condition (L), i.e., 
\[
(L)\Longrightarrow (F)\quad\text{as}\quad n\to\infty.
\]
\end{proposition}
\begin{proof}
By the expression $ \sum\limits_{j=1}^{k_{n}}\sigma^{2}_{n, j}=1,$ for any $\epsilon>0,$ we have 
\begin{equation}\label{equ:2.1}
\begin{split}
\max\limits_{1\leq j\leq k_{n}}\sigma^{2}_{n, j}&\leq \sum\limits_{j=1}^{k_{n}}\int\limits_{-\infty}^{\infty}x^{2}dF_{n, j}(x) \leq \sum\limits_{j=1}^{k_{n}}\int\limits_{|x|<\epsilon}x^{2}dF_{n, j}(x)  + \sum\limits_{j=1}^{k_{n}}\int\limits_{|x| \geq \epsilon}x^{2}dF_{n, j}(x)\\
&\leq \epsilon^{2}\sum\limits_{j=1}^{k_{n}}\sigma^{2}_{n, j}+ \sum\limits_{j=1}^{k_{n}}\int\limits_{|x| \geq \epsilon}x^{2}dF_{n, j}(x)\leq \epsilon^{2}+ \sum\limits_{j=1}^{k_{n}}\int\limits_{|x| \geq \epsilon}x^{2}dF_{n, j}(x).
\end{split}
\end{equation}
Applying the Lindeberg condition (L), from (\ref{equ:2.1}), the proof is finished  since $\epsilon>0$ is an arbitrary small.
\end{proof}
 Moreover, a condition is stronger than (L), as presented in the following definition.
 \begin{definition}\label{def:2.3}\quad A double sequence $\{X_{n, j}\}$ is said  to satisfy the Lyapunov's condition if
\begin{equation*}
\sum\limits_{j=1}^{k_{n}}\int\limits_{-\infty}^{+\infty}x^{2+\delta}dF_{n, j}(x)=o(1) \quad\text{as}\quad n\to\infty, \tag{$\Lambda$}
\end{equation*}
for $\delta\in (0, 1].$
\end{definition}
\begin{proposition}\label{pro:2.5}\quad The Lindeberg's condition (L) is deduced by the Lyapunov's condition ($\Lambda$), i.e.,
\[
(\Lambda)\Longrightarrow (L).
\] 
\end{proposition}
\begin{proof}
 Indeed,  for  any $\epsilon >0,$ since $|X_{n, j}|\geq \epsilon$ implies $|X_{n, j}|^{\delta}\geq \epsilon ^{\delta},$ for $\delta\in (0,1].$\\
  Therefore
\begin{equation}\label{equ:2.2}
\sum\limits_{j=1}^{k_{n}}\int\limits_{|x|\geq\epsilon}x^{2}dF_{n, j}(x)\leq \frac{1}{\epsilon^{\delta}}\sum\limits_{j=1}^{k_{n}}\int\limits_{|x|\geq\epsilon}x^{2+\delta}dF_{n, j}(x)\leq \frac{1}{\epsilon^{\delta}}\sum\limits_{j=1}^{k_{n}}\int\limits_{-\infty}^{+\infty}x^{2+\delta}dF_{n, j}(x).
\end{equation}
Apply ($\Lambda$), from (\ref{equ:2.2}), the proof is complete.
\end{proof}
The preceding conditions, such as ($\Lambda$), (L), (F), and  (I), play an important role in the validity of the central limit theorems of the classical situation, which will be considered in the next parts.

\subsection{Classical central limit theorems} 
 The next statements are the central limit theorems for the double sequence $\{X_{n, j}\}$ in the classical version. The detailed proofs may be found in \cite{Rotar1996} and \cite{Shiryaev1996}. From now on, assume that the $\{X_{n, j}\}$ is a double sequence of independent in each row infinitesimal random variables $X_{n, j},$ such that the array conditions (A), (B), and (C) in the first section are suggested.
\begin{theorem}\label{the:2.1}(Lindeberg's theorem)\quad Suppose that a sequence $\{X_{n, j}\}$ satisfies the condition Lindeberg (L). Then the central limit theorem (CLT) is valid, i.e.,
\[
(L) \Longrightarrow (CLT).
\] 
\end{theorem}
From Proposition (\ref{pro:2.5}),   we have the following conclusion:
\begin{theorem}\label{the:2.2}(Lyapunov's theorem)\quad The sequence $\{X_{n, j}\}$ obeys the (CLT) if the condition ($\Lambda$) holds. It means that 
\[
(\Lambda)  \Longrightarrow (CLT).
\] 
\end{theorem}
From \cite{Feller1971} (Chapter XV, Theorem 2, Page 520),  the  following statement states that the Lindeberg condition is necessary condition for the central limit theorem (CLT) under the Feller condition (F).
\begin{theorem}\label{the:2.3}(Feller's theorem)\quad Suppose that the sequence $\{X_{n, j}\}$ obeys the central limit theorem (CLT) under the Feller condition (F). Then the Lindeberg condition is valid, i.e., 
\[
(CLT)\quad \& \quad (F) \Longrightarrow (L).
\]
\end{theorem}

\subsection{Non-classical central limit theorems}
Following Zolotarev \cite{Zolotarev1997}, the limit theorems that make no use of the condition of asymptotic infinitesimality (I) are said to be non-classical. The preceding listed central limit theorems, like Theorems \ref{the:2.1}, \ref{the:2.2}, and \ref{the:2.3},  are classical.  As was noted in some monographs containing the main results of the  non-classical theory of summation of independent random variables, the ideas underlying the non-classical approach go back to P. L{\'e}vy, who studied various versions of the central limit theorem without the assumption of  condition of asymptotic infinitesimality of the individual summands $X_{n, j}$ for $ngeq 1, 1\leq j\leq k_{n}.$ A simple example introduced by Shiryaev (\cite{Shiryaev1996}, Page 337) shows that a double sequence $\{X_{n, j}\}$ is formed by the sequence of independent (in each row) normally distributed random variables $(X_{j}, 1\leq j\leq k_{n}),$ using the "series form" $X_{n, j}=\frac{X_{j}}{B_{n}},$  for which 
\[
\mathbb{E} X_{j}=0, \mathbb{D}X_{1}=1, \mathbb{D}X_{j}=2^{j-2}, (2\leq j\leq k_{n});\quad\text{and}\quad  B^{2}_{n}=\sum\limits_{j=1}^{k_{n}}\mathbb{D}X_{j}
\]
do not use the conditions for the central limit theorems in classical version like Lindeberg's, Feller's,  and (I) conditions but the sequence $\{X_{n, j}\}$ still satisfies the central limit theorem. Indeed, for the original sequence $(X_{j}, 1\leq j\leq k_{n})$ let us set $W_{n}=\sum\limits_{j=1}^{k_{n}}X_{j},$ where $X_{j}\stackrel{D}{\sim}\mathcal{N}(0, 2^{j-2})$ for $2\leq j\leq k_{n}.$  For $k_{n}\geq 2, $ it is easily to check that  
\[
\mathbb{E} W_{n}=\sum\limits_{j=1}^{k_{n}}\mathbb{E} X_{j}=0, \quad\text{and}\quad \mathbb{D} W_{k_{n}}=B^{2}_{n}=\sum\limits_{j=1}^{k_{n}}\mathbb{D}X_{j}=1+1+\sum\limits_{j=3}^{k_{n}}\mathbb{D}X_{j}=2^{k_{n}-1} (\ne 1).
\]
Therefore, $W_{n}\stackrel{D}{\sim}\mathcal{N}(0, 2^{k_{n}-1})$ for $k_{n}\geq 2.$ For the double sequence $\{X_{n, j}\},$ we still use the same symbol as before $S_{n}=\sum\limits_{j=1}^{k_{n}}X_{n, j}, $ where $X_{n, j}=\frac{X_{j}}{B_{n}}, (1\leq j\leq k_{n}).$ We see  at once that, for $k_{n}\geq 1$ 
\[
\mathbb{E} S_{n}= \frac{1}{B_{n}}\mathbb{E} {W}_{k_{n}}=0, \quad \mathbb{D}S_{n}=  \frac{1}{B_{n}}\mathbb{D} {W}_{n}=\frac{1}{2^{k_{n}-1}}2^{k_{n}-1}=1.
\]
It is obvious that $S_{n}\stackrel{D}{\sim}\mathcal{N}(0, 1)$ for $k_{n}\geq 1.$ And this means that the central limit theorem for the double sequence $\{X_{n, j}\}$ is valid.\\
 However, it is easy to check that the double sequence $\{X_{n, j}\}$ does not satisfy any of the conditions (L), (F), and (I). Indeed, (F) is not satisfied since
\[
\max\limits_{1\leq j\leq k_{n}}\sigma^{2}_{n, j}=\max\limits_{1\leq j\leq k_{n}}\frac{\sigma^{2}_{j}}{B^{2}_{n}}=\frac{2^{k_{n}-2}}{2^{k_{n}-1}}=\frac{1}{2}\nrightarrow 0 \quad\text{as}\quad n\to\infty.
\] 
Thus, the conditions (L) and (I) are also not satisfied for the double sequence $\{X_{n, j}\}$ which means that the central limit theorem in the classical situation is not satisfied. But it is true for the central limit theorem in its non-classical situation. In nature, the sequence in the above example meets Rotar's condition below.
\begin{definition}\label{def:2.4}
 A sequence of series $\{X_{n, j}\}$ is said to satisfy the Rotar condition if for any $\epsilon >0$
\begin{equation*}
\sum\limits_{j=1}^{k_{n}}\int\limits_{|x|>\epsilon}|x||F_{n, j}(x)-\Phi_{n, j}(x)|dx=o(1)\quad\text{as}\quad n\to\infty,\tag{R}
\end{equation*}
here $\Phi_{n, j}(x)$ denotes the normal distribution of a normal random variable $X^{*}_{n, j}\stackrel{D}{\sim}\mathcal{N}(0, \sigma^{2}_{n, j}).$
\end{definition}
\begin{remark}\label{rem:2.1}\quad  It is evident that
\[
X^{*}_{n, j}\stackrel{D}{=}\sigma_{n, j}X^{*},
\]
and
\[
\Phi_{n, j}(x)=P(X^{*}_{n, j}<x)=P(\sigma_{n, j}X^{*}< x)=P(X^{*}< x/\sigma_{n, j})=\Phi_{0,1} (x/\sigma_{n, j}),
\]
for  $1\leq j\leq k_{n}, n\geq 1.$
\end{remark}
The following result  concludes that the Rotar's condition (R) is weaker than the Lindeberg's condition (L). 
\begin{proposition}\label{pro:2.6}\quad For any $\epsilon >0,$ the following implication
\[
(L)\Longrightarrow (R)
\]
holds.
\end{proposition}
\begin{proof}
From \cite{Formanov2019} (Theorem 1, Formula (3), Page 32),  we have  the following estimation
\[
\sum\limits_{j=1}^{k_{n}}\int\limits_{|x|>\epsilon}|x||F_{n, j}(x)-\Phi_{n, j}(x)|dx
\leq C\bigg(\sum\limits_{j=1}^{k_{n}}\int\limits_{|x|> \epsilon}x^{2}dF_{n, j}(x)+\sum\limits_{j=1}^{k_{n}}\sigma^{2s}_{n, j}\bigg), 
\]
for some $C=C(\epsilon)>0$ and for any $\epsilon>0.$ With $s= 2,$ from above inequality, it follows that
\begin{equation}\label{equ:2.3}
\sum\limits_{j=1}^{k_{n}}\int\limits_{|x|>\epsilon}|x||F_{n, j}(x)-\Phi_{n, j}(x)|dx
\leq C\bigg(\sum\limits_{j=1}^{k_{n}}\int\limits_{|x|> \epsilon}x^{2}dF_{n, j}(x)+\sup\limits_{1\leq j \leq k_{n}}\sigma^{2}_{n, j}\bigg).
\end{equation}
It should be noted that, to obtain the estimate (\ref{equ:2.3}),  we have used the fact that 
\[
\sum\limits_{j=1}^{k_{n}}\sigma^{4}_{n, j}\leq \sum\limits_{j=1}^{k_{n}}\sigma^{2}_{n, j}\bigg( \max\limits_{1\leq j\leq k_{n}}\sigma^{2}_{n, j}\bigg)=\max\limits_{1\leq j\leq k_{n}}\sigma^{2}_{n, j},
\]
 since $\sum\limits_{j=1}^{k_{n}}\sigma^{2}_{n, j}=1.$ Finally, from (\ref{equ:2.3}), the Rotar's condition (R) is deduced by the Lindeberg's condition (L) when $n\to\infty.$ Meanwhile, $\sup\limits_{1\leq j\leq k_{n}}\sigma^{2}_{n, j}=o(1)$ as $n\to\infty$ since (L) holds. The proof is complete.
\end{proof}
 The following Rotar's theorem shows that  the Rotar's condition (R) is sufficient and necessary for the validity of the central limit theorem (CLT).
\begin{theorem}\label{the:2.4}(Rotar's theorem)\quad To have 
\[
\Delta_{n, k_{n}}=o(1)\quad\text{as}\quad n\to\infty,
\]
it is sufficient (and necessary) that for every $\epsilon >0$ the condition (R) is satisfied.
\end{theorem}
 For a deeper discussion of the non-classical central limit theorem like  the Rotar's theorem we refer the reader to \cite{Formanov2010},  \cite{Formanov2019}, \cite{Formanov2022}, \cite{Rotar1975}, \cite{Rotar1996},  \cite{Shiryaev1996}, and \cite{Zolotarev1997}.  

\section{Preliminaries}\label{sec:3}
From now on,  $X^{*}$ denotes the standard normal distributed random variable with zero mean, unit variance, and characteristic function $f_{X^{*}}(t)=e^{-\frac{1}{2}t^{2}},$ which is denoted by $X^{*}\stackrel{D}{\sim} \mathcal{N}(0,1).$  Let $\{X^{*}_{n, j}\}:=\{X^{*}_{n, j}: n\geq 1, 1\leq j\leq k_{n}\}$ be a sequence of   independent normally distributed random variables with zero means and finite variances $\sigma^{2}_{n, j},$ which are denoted by $X^{*}_{n, j}\stackrel{D}{\sim} \mathcal{N}(0, \sigma^{2}_{n, j}).$ The following representation of the standard normal distributed random variable $X^{*}$  is necessary for the proofs in the next parts. 
 \begin{lemma}\label{lem:3.1}\quad  Let $(\nu_{n}, n\geq 1)$ be a sequence integer-valued positive random variables, independent of all the random variables $X^{*}_{n, j}$ for $n\geq 1$ and $1\leq j\leq k_{n}.$  Then, for $n\geq 1$   
 \[
 X^{*}\stackrel{D}{=}\sum\limits_{j=1}^{\nu_{n}}X^{*}_{n, j}.
 \]
 \end{lemma}
 \begin{proof}\quad From \cite{Gut2005}, using the Remark \ref{rem:2.1}, since the independence of $X^{*}_{n, j}$ for $n\geq 1, 1\leq j\leq k_{n}$ and $\sum\limits_{j=1}^{k_{n}}\sigma^{2}_{n, j}=1,$  it follows that 
 \[
 \begin{split}
&f_{\sum\limits_{j=1}^{\nu_{n}}X^{*}_{n, j}}(t) = \sum\limits_{k_{n}=1}^{\infty}P(\nu_{n}=k_{n})f_{\sum\limits_{j=1}^{k_{n}}X^{*}_{n, j}}(t)=\sum\limits_{k_{n}=1}^{\infty}P(\nu_{n}=k_{n})\prod\limits_{j=1}^{k_{n}}f_{X^{*}_{n, j}}(t)\\
&=\sum\limits_{k_{n}=1}^{\infty}P(\nu_{n}=k_{n})e^{-\frac{1}{2}t^{2}\sum\limits_{j=1}^{k_{n}}\sigma^{2}_{n, j}}=e^{-\frac{1}{2}t^{2}}=f_{X^{*}}(t).
\end{split}
 \]
 The above quality shows that the characteristic function of the random variable $\sum\limits_{j=1}^{\nu_{n}}X^{*}_{n, j}$ coincides with the characteristic function of the standard normally distributed random variable $X^{*}.$  This finishes the proof.
 \end{proof}
By \cite{Zolotarev1976}, we recall the concept of a probability distance, which will be used in the next sections. Let us denote by $\mathfrak{X}$ the set of random variables on a probability space $(\Omega, \mathcal{A}, \mathbb{P}).$ 
\begin{definition}\label{def:3.1}\quad Assume that the random variables $X, Y, Z\in\mathfrak{X}.$ The mapping  $d: \mathfrak{X}\times \mathfrak{X}\longrightarrow [0, +\infty]$ is called a probability distance in $\mathfrak{X},$ denoted by $d(X, Y),$ if the following statements hold:
\begin{enumerate}
\item[(1).] $d(X, Y)=0$ if $X\stackrel{a.s.}{=}Y,$
\item[(2).] $d(X, Y)=d(Y, X),$
\item[(3).] $d(X, Y)\leq d(X, Z)+d(Z, Y).$
\end{enumerate}
\end{definition}
\begin{definition}\label{def:3.2}\quad A metric $d$ is called simple if its values are determined by a pair of marginal distributions $P_{X}$ and $P_{Y}.$ In all other cases $d$ is called composed.
\end{definition}
It should be noted that, for a simple metric the following forms are equivalent
\[
d(X, Y)=d(P_{X}, P_{Y})=d(F_{X}, F_{Y}).
\]
\begin{definition}\label{def:3.3}\quad A simple distance $d$ in $\mathfrak{X}$ is called an ideal probability distance of order $s\geq 0,$ if the following statements hold:
\begin{enumerate}
\item[(i).] (Regularity) For the independent random variables $X, Y, Z\in\mathfrak{X}$
\[
d(X+Z, Y+Z)\leq d(X, Y).
\]
\item[(ii).] (Homogeneity of order $s$) For $c\ne 0,$ and $s\geq 0,$
\[
d(cX, cY)=|c|^{s}d(X, Y).
\]
\end{enumerate}
\end{definition}
It is worth pointing out that an interesting consequence of the regularity and homogeneity properties is
the semi-additivity of the metric $d:$ \quad Let $X_{1}, X_{2}, \dots , X_{n}$ and $Y_{1}, Y_{2}, \ldots , Y_{n}$ be two collections of independent random variables, then for real numbers $c_{j}, 1\leq j\leq n$ and for $s\geq 0$
\[
d\bigg(\sum\limits_{j=1}^{n}X_{j}, \sum\limits_{j=1}^{n}Y_{j} \bigg)\leq \sum\limits_{j=1}^{n}|c_{j}|^{s}d(X_{j}, Y_{j}).
\]
Next, we introduce the Zolotarev distance as an ideal probability distance that is widely applied in the theory of summation of independent random variables (see \cite{Zolotarev1997} and the references given there). The definition of Zolotarev’s distance is presented  as follows:
\begin{definition}\label{def:3.4}(\cite{Zolotarev1983})\quad The Zolotarev distance of order $s (s>0)$ on $\mathfrak{X}$ between two random variables $X, Y \in \mathfrak{X}$ is defined by
\[
\zeta_{s}(X, Y):=\sup\limits_{f\in\mathcal{D}_{s}}\bigg|\mathbb{E}f(X)-\mathbb{E}f(Y) \bigg|,
\]
here $\mathcal{D}_{s}:=\{f\in C^{r}(\mathbb{R}): |f^{(r)}(x)-f^{(r)}(y)|\leq |x-y|^{\beta}\}, r\in\mathbb{N}, \beta\in (0,1], s=r+\beta.$ The $f^{(r)}$ stands for the rth derivative of function $f.$ 
\end{definition}
Here we denote by $C(\mathbb{R})$ the set of all real-valued, bounded, and uniformly continuous functions on $\mathbb{R}$ with norm $||f||=\sup\limits_{x\in\mathbb{R}}|f(x)|$ and $C^{r}(\mathbb{R}):=\{f\in C(\mathbb{R}): f^{(k)}\in C(\mathbb{R}), 1\leq k\leq r\}.$

The main properties of the Zolotarev's distance will be presented below. The detailed proofs may be found in \cite{Zolotarev1976}, \cite{Zolotarev1983}, and \cite{Zolotarev1997}. 
\begin{proposition}\label{pro:3.1}(\cite{Zolotarev1997})\quad The distance $\zeta_{s}(X, Y)$ on $\mathfrak{X}$ is an ideal probability metric of order $s>0,$ that is, for any $c\ne 0$ and for the independent random variables $X, Y, Z\in \mathfrak{X},$ we have
\begin{enumerate}
\item[{i.}]  $\zeta_{s}(X+Z, Y+Z)\leq \zeta_{s}(X, Y),$
\item[{ii.}] $\zeta_{s}(cX, cY)=|c|^{s}\zeta_{s}(X, Y).$
\end{enumerate}
\end{proposition}
\begin{proposition}\label{pro:3.2}\quad  A sufficient condition for the validity of the random central limit theorem (RCLT)
\[
\Delta_{n, \nu_{n}}=o(1), \quad\text{as}\quad n\to\infty
\]
is 
\[
\zeta_{s}(S_{n, \nu_{n}}, X^{*})=o(1), \quad\text{as}\quad n\to\infty,
\]
for $s>0.$
\end{proposition}
\begin{proof}
From \cite{Zolotarev1976}, the convergence of the $\zeta_{s}(S_{n, k_{n}}, X^{*})=o(1)$ follows the $\Delta_{n, k_{n}}=o(1),$  when $n\to\infty.$ Therefore, on account of \cite{Gut2005}, it follows that if
\[
\zeta_{s}(S_{n, \nu_{n}}, X^{*})=o(1), \quad\text{as}\quad n\to\infty,
\]
holds, then 
\[
\Delta_{n, \nu_{n}}=o(1), \quad\text{as}\quad n\to\infty
\]
is satisfied, i. e.,  (RCLT) is valid.
\end{proof}
\begin{proposition}\label{pro:3.3}(\cite{Zolotarev1983}, Formula (4.3), Page 277)\quad Let $\{X_{n, j}\}$ and $\{Y_{n, j}\}$ be two double sequences of independent (in each row) random variables  for $n\geq 1$ and $1\leq j\leq k_{n}.$ Then, for $n\geq 1$ and $s>0$
\[
\zeta_{s}\bigg(\sum\limits_{j=1}^{k_{n}}X_{n, j}, \sum\limits_{j=1}^{k_{n}}Y_{n, j} \bigg)\leq \sum\limits_{j=1}^{k_{n}}\zeta_{s}\bigg(X_{n, j}, Y_{n, j} \bigg).
\]
\end{proposition} 
From Proposition (\ref{pro:3.3}), it is follows that
\begin{proposition}\label{pro:3.4}\quad Let $\{X_{n, j}\}$ and $\{Y_{n, j}\}$ be two double sequences of i.i.d. (in each row) random variables for $n\geq 1$ and $1\leq j\leq k_{n}.$ Then, for $n\geq 1$ and $s>0$
\[
\zeta_{s}\bigg(\sum\limits_{j=1}^{k_{n}}X_{n, j}, \sum\limits_{j=1}^{k_{n}}Y_{n, j} \bigg)\leq k_{n} \zeta_{s}\bigg(X_{n, 1}, Y_{n, 1} \bigg).
\]
\end{proposition} 
Following \cite{Gut2005}, we conclude that
\begin{proposition}\label{pro:3.5}\quad Let $\{X_{n, j}\}$ and $\{Y_{n, j}\}$ be two double sequences of independent (in each row) random variables  for $n\geq 1$ and $1\leq j\leq k_{n}.$ Assume that the sequence $(\nu_{n}, n\geq 1)$ is  independent of two double sequences $\{X_{n, j}\}$ and $\{Y_{n, j}\}$ for $n\geq 1$ and $1\leq j\leq k_{n}.$ Then, for $n\geq 1$ and $s>0$
\[
\zeta_{s}\bigg(\sum\limits_{j=1}^{\nu_{n}}X_{n, j}, \sum\limits_{j=1}^{\nu_{n}}Y_{n, j} \bigg)\leq \sum\limits_{k_{n}=1}^{\infty}P(\nu_{n}=k_{n})\sum\limits_{j=1}^{k_{n}}\zeta_{s}\bigg(X_{n, j}, Y_{n, j} \bigg).
\]
\end{proposition} 
From Proposition (\ref{pro:3.5}), we have
\begin{proposition}\label{pro:3.6}\quad Let $\{X_{n, j}\}$ and $\{Y_{n, j}\}$ be two double sequences of i.i.d. (in each row) random variables for $n\geq 1$ and $1\leq j\leq k_{n}.$ Suppose that the sequence $(\nu_{n}, n\geq 1)$ is  independent of two double sequences $\{X_{n, j}\}$ and $\{Y_{n, j}\}$ for $n\geq 1$ and $1\leq j\leq k_{n}.$ Then, for $n\geq 1$ and $s>0$ 
\[
\zeta_{s}\bigg(\sum\limits_{j=1}^{\nu_{n}}X_{n, j}, \sum\limits_{j=1}^{\nu_{n}}Y_{n, j} \bigg)\leq \mathbb{E}(\nu_{n}) \zeta_{s}\bigg(X_{n, 1}, Y_{n, 1} \bigg).
\]
\end{proposition} 

\section{Conditions for the validity of the central limit theorems for random sums}\label{sec:4}
The conditions in this section will play an important role in determining the validity of the random central limit theorems for random sums. Those conditions are generalizations of the conditions presented in the previous part in the sense that if $\nu_{n}\stackrel{a.s}{=}k_{n}$ for $n\geq 1,$ the random conditions deduce the conditions in non-random versions.
\begin{definition}\label{def:4.1}\quad The double sequence $\{X_{n, j}\}$ is said to satisfy the random Lindeberg condition  if for any $\epsilon >0$
\begin{equation*}
\mathbb{E}\bigg\{\sum\limits_{j=1}^{\nu_{n}}\mathbb{E}\bigg[X^{2}_{n, j}\mathbb{I}(|X_{n, j}|>\epsilon) \bigg]\bigg\}=o(1)\quad\text{as}\quad n\to\infty.\tag{RL}
\end{equation*}
\end{definition}
The following condition is stronger than random Lindeberg condition (RL).
\begin{definition}\label{def:4.2}\quad The double sequence $\{X_{n, j}\}$ is said to satisfy the random Lyapunov condition  if 
\begin{equation*}
\mathbb{E}\bigg\{\sum\limits_{j=1}^{\nu_{n}}\mathbb{E}\bigg(|X_{n, j}|^{2+\delta}\bigg)\bigg\}=o(1) \quad\text{as}\quad n\to\infty, \tag{$R\Lambda$}
\end{equation*}
for $\delta\in (0, 1].$ 
\end{definition}
 We claim that the random Lindeberg condition (RL) is deduced by the random Lyapunov condition ($R\Lambda$). 
 \begin{proposition}\label{pro:4.1}\quad For $\epsilon >0$ and for $\delta\in (0,1],$ we have
 \[
 (R\Lambda)\Longrightarrow (RL).
 \]
 \end{proposition}
 \begin{proof}
 From the inequality (\ref{equ:2.2}), for  $\epsilon >0$ and for $\delta\in (0,1],$ it follows that
\begin{equation}\label{equ:4.1}
\mathbb{E}\bigg\{\sum\limits_{j=1}^{\nu_{n}}\mathbb{E}\bigg(X^{2}_{n, j}\mathbb{I}(|X_{n, j}|\geq \epsilon) \bigg)\bigg\}\leq \frac{1}{\epsilon^{\delta}}\mathbb{E}\bigg\{\sum\limits_{j=1}^{\nu_{n}}\mathbb{E}\bigg(X^{2+\delta}_{n, j}\mathbb{I}(|X_{n, j}|\geq \epsilon) \bigg)\bigg\}\leq \frac{1}{\epsilon^{\delta}}\mathbb{E}\bigg\{\sum\limits_{j=1}^{\nu_{n}}\mathbb{E}\bigg(X^{2+\delta}_{n, j} \bigg)\bigg\}.
\end{equation}
Applying $(R\Lambda),$ we have the proof.
 \end{proof}
 The following condition is weaker than random Lindeberg condition (RL).
 \begin{definition}\label{def:4.3}\quad  The double sequence $\{X_{n, j}\}$ is said to satisfy the random Feller condition  if 
\begin{equation*}
\mathbb{E}\bigg\{\max\limits_{1\leq j\leq \nu_{n}}\sigma^{2}_{n, j}\bigg\}=o(1)\quad\text{as}\quad n\to\infty. \tag{RF}
\end{equation*}
 \end{definition}
 The next implication shows that the random Feller condition is deduced by the random Lindeberg condition. 
\begin{proposition}\label{pro:4.2}\quad For $\epsilon >0,$ we have
 \[
(RL)\Longrightarrow (RF), \quad\text{as}\quad n\to\infty.
\]
 \end{proposition}
\begin{proof}
From inequality (\ref{equ:2.1}), it is easy to follow that the following estimation 
\begin{equation}\label{equ:4.2}
\mathbb{E}\bigg\{\max\limits_{1\leq j\leq \nu_{n}}\sigma^{2}_{n, j}\bigg\}\leq \epsilon^{2}+\mathbb{E}\bigg\{\max\limits_{1\leq j\leq \nu_{n}}\int\limits_{|x|>\epsilon}x^{2}dF_{n, j}(x)\bigg\}
\leq \epsilon^{2}+\mathbb{E}\bigg\{\sum\limits_{j=1}^{\nu_{n}}\int\limits_{|x|>\epsilon}x^{2}dF_{n, j}(x)\bigg\}
\end{equation}
is true, for any $\epsilon >0.$ 
\end{proof}
The inequality (\ref{equ:4.2}) shows that for any double sequence $\{X_{n, j}\}$ of independent random variables (in each row) which have zero means $\mathbb{E} X_{n, j}=0$ and finite variances $0<\mathbb{D} X_{n, j}=\sigma^{2}_{n, j}<\infty,$ the random Lindeberg condition deduces the random Feller condition. The reverse is true for the sequence $\{X^{*}_{n, j}\}$ of independent (in each row) normally distributed random variables which  have zero means $\mathbb{E} (X^{*}_{n, j})=0$ and positive finite variances $\mathbb{D}(X^{*}_{n, j})=\sigma^{2}_{n, j}\in (0, +\infty).$
\begin{proposition}\label{pro:4.3}\quad If  the random Feller condition (RF) is satisfied for  the sequence $\{X^{*}_{n, j}\},$ then the random Lindeberg condition (RL) is also valid. This means that, for the sequence $\{X^{*}_{n, j}\},$ the following implication:
\[
(RF)\Longrightarrow (RL), \quad\text{as}\quad n\to\infty,
\]
holds.
\end{proposition}
\begin{proof}
For $\epsilon>0, n\geq 1$ and with the suggestion that $\sum\limits_{j=1}^{k_{n}}\sigma^{2}_{n, j}=1,$ the Lindeberg condition (L) for the sequence $\{X^{*}_{n, j}\}$ shows that
\[
\begin{split}
\sum\limits_{j=1}^{k_{n}}\mathbb{E}\bigg[(X^{*}_{n, j})^{2}\mathbb{I}(|X^{*}_{n, j})|\geq \epsilon \bigg]&=\sum\limits_{j=1}^{k_{n}}\mathbb{E}\bigg[(\sigma_{n, j}X^{*})^{2}\mathbb{I}(|\sigma_{n, j}X^{*})|\geq \epsilon \bigg]\\
&=\sum\limits_{j=1}^{k_{n}}\sigma^{2}_{n, j}\mathbb{E}\bigg[(X^{*})^{2}\mathbb{I}(|X^{*})|\geq \frac{\epsilon}{\sigma_{n, j}} \bigg]\leq\mathbb{E}\bigg[(X^{*})^{2}\mathbb{I}(|X^{*})|\geq \frac{\epsilon}{\sigma^{*}} \bigg],
\end{split}
\]
where $\sigma^{*}=\sup\limits_{1\leq j\leq k_{n}}\sigma_{n, j}.$\\
Therefore, from \cite{Gut2005}, we have the following estimation:
\begin{equation}\label{equ:4.3}
\mathbb{E}\bigg\{\sum\limits_{j=1}^{\nu_{n}}\mathbb{E}\bigg[(X^{*}_{n, j})^{2}\mathbb{I}(|X^{*}_{n, j})|\geq \epsilon \bigg]\bigg\}\leq \mathbb{E}\bigg[(X^{*})^{2}\mathbb{I}(|X^{*})|\geq \frac{\epsilon}{\sigma^{*}}\bigg],
\end{equation}
for $\epsilon>0$ and $\sigma^{*}=\sup\limits_{1\leq j\leq k_{n}}\sigma_{n, j}.$\\
If the random Feller condition (RF) is satisfied for the sequence  $\{X^{*}:=X^{*}_{n, j}/\sigma_{n, j}\},$ then $\frac{1}{\sigma^{*}}\to \infty$ as $n\to\infty.$ It is conclude  that 
\[
\mathbb{E}\bigg[(X^{*})^{2}\mathbb{I}(|X^{*})|\geq \frac{\epsilon}{\sigma^{*}}\bigg]=o(1)\quad\text{as}\quad n\to\infty,
\]
since $\mathbb{E}(X^{*})^{2}=1.$ The above inequality (\ref{equ:4.3}) shows that the random Lindeberg  condition (RL) for the sequence  $\{X^{*}\}$ is valid. The proof  is complete.
\end{proof}
The following condition divides the family of central limit theorems for random sums into two classes: classical and non-classical. It is so-called the random condition of asymptotic infinitesimality for independent individual summands  $X_{j}$ in the random sums $\sum\limits_{j=1}^{\nu_{n}}X_{n, j}.$ 
\begin{definition}\label{def:4.4}(Random condition of asymptotic infinitesimality)\quad The individual summands $X_{n, j} (n \geq 1, 1\leq j\leq k_{n})$ in random sums $\sum\limits_{j=1}^{\nu_{n}}X_{n, j}$   are said to satisfy the random condition of asymptotic infinitesimality if, for any $\epsilon >0$ 
\begin{equation*}
\mathbb{E}\bigg\{\max\limits_{1\leq j\leq \nu_{n}}P\big(|X_{n, j}|\geq \epsilon\big) \bigg\}=o(1)\quad\text{as}\quad n\to\infty.\tag{RI}
\end{equation*} 
\end{definition}
\begin{remark}\label{rem:4.1}\quad If $\nu_{n}\stackrel{a.s.}{=}k_{n}$ for $n\geq 1,$ then the random condition of asymptotic infinitesimality (RI) deduces the non-random condition of asymptotic infinitesimality (I).
\end{remark}
The following propositions establish the relation between the random condition of asymptotic infinitesimality for the individual summands $X_{n, j},  (n\geq 1, 1\leq j\leq k_{n}),$ and the related conditions in the central limit theorem of classical situations.
 \begin{proposition}\label{pro:4.4}\quad The following statements are equivalent:
 \begin{enumerate}
 \item[({\bf A}).] The individual summands $X_{n, j} (n\geq 1, 1\leq j\leq k_{n})$ are satisfied the (RI), that is for any $\epsilon>0$
 \[
 \mathbb{E}\bigg\{\max\limits_{1\leq j\leq \nu_{n}}P\big(|X_{n, j}|\geq \epsilon\big) \bigg\}=o(1)\quad\text{as}\quad n\to\infty.
 \]
 \item[({\bf B}).] As $n\to\infty$
 \[
 \mathbb{E}\bigg\{\max\limits_{1\leq j\leq \nu_{n}}\mathbb{E}\bigg(\frac{X_{n, j}^{2}}{1+X_{j}^{2}}\bigg)\bigg\}=o(1).
 \]
 \item[({\bf C}).] Denote  $f_{n, j}(t)=\mathbb{E}(e^{itX_{n, j}})$ the characteristic function of a random variable $X_{n, j}.$ Then
 \[
 \mathbb{E}\bigg\{\max\limits_{1\leq j\leq \nu_{n}}\big|f_{n, j}(t)-1 \big|\bigg\}=o(1) \quad\text{as}\quad n\to\infty.
 \]
 \end{enumerate}
\end{proposition}
\begin{proof} $({\bf A})\Longrightarrow ({\bf B}):$ Let $\epsilon>0$ be a chosen small number and $k_{n}$ be a sufficient large. Then
\[
\begin{split}
&\max\limits_{1\leq j\leq k_{n}}\mathbb{E}\bigg(\frac{X_{n, j}^{2}}{1+X_{n, j}^{2}}\bigg)\leq \max\limits_{1\leq j\leq k_{n}}\bigg[\mathbb{E}\bigg(\frac{X_{n, j}^{2}}{1+X^{2}_{n, j}}\mathbb{I}(|X_{n, j}|<\epsilon)\bigg)\bigg]\\
&+\max\limits_{1\leq j\leq k_{n}}\bigg[\mathbb{E}\bigg(\mathbb{I}(|X_{n, j}|\geq \epsilon)\bigg)\bigg]\leq \frac{\epsilon^{2}}{1+\epsilon^{2}}+\max\limits_{1\leq j\leq k_{n}}P\bigg(|X_{n, j}|\geq\epsilon\bigg).
\end{split}
\] 
Therefore
\[
\mathbb{E}\bigg\{\max\limits_{1\leq j\leq \nu_{n}}\mathbb{E}\bigg(\frac{X_{n, j}^{2}}{1+X_{j}^{2}}\bigg)\bigg\}\leq  \frac{\epsilon^{2}}{1+\epsilon^{2}}+\mathbb{E}\bigg(\max\limits_{1\leq j\leq \nu_{n}}P(|X_{n, j}|\geq \epsilon) \bigg).
\]
Since $\epsilon>0$ is chosen arbitrarily small, the last inequality shows that $(A)\Longrightarrow (B).$\\
$({\bf B})\Longrightarrow ({\bf A}):$ For every $\epsilon >0,$ let us consider  the following inequalities:
\[
\begin{split}
&\max\limits_{1\leq j\leq k_{n}}\mathbb{E}\bigg(\frac{X_{n, j}^{2}}{1+X_{n, j}^{2}}\bigg)\geq \max\limits_{1\leq j\leq k_{n}}\mathbb{E}\bigg(\frac{X_{n, j}^{2}}{1+X_{n, j}^{2}}\mathbb{I}(|X_{n, j}|\geq \epsilon)\bigg)\\
&\geq \frac{\epsilon^{2}}{1+\epsilon^{2}}\max\limits_{1\leq j\leq k_{n}}\mathbb{E}\bigg(\mathbb{I}(|X_{n, j}|\geq \epsilon)\bigg)=\frac{\epsilon^{2}}{1+\epsilon^{2}}\max\limits_{1\leq j\leq k_{n}}P\bigg(|X_{n, j}|\geq \epsilon\bigg).
\end{split}
\]
Then, our assertion will be valid via the following inequality:
\[
\mathbb{E}\bigg(\max\limits_{1\leq j\leq \nu_{n}}P(|X_{n, j}|\geq \epsilon)\bigg)\leq  \mathbb{E}\bigg\{\max\limits_{1\leq j\leq \nu_{n}}\mathbb{E}\bigg(\frac{X_{n, j}^{2}}{1+X_{n, j}^{2}}\bigg)\bigg\}.
\]
$({\bf A})\Longrightarrow ({\bf C}):$ For every $\epsilon >0,$ let us discuss the following inequalities:
\[
\begin{split}
\max\limits_{1\leq j\leq k_{n}}|f_{n, j}(t)-1| &=\max\limits_{1\leq j\leq k_{n}}\bigg|\int\limits_{-\infty}^{+\infty}(e^{itx}-1)dF_{n, j}(x)\bigg|\\
&\leq \max\limits_{1\leq j\leq k_{n}}\int\limits_{|x|<\epsilon}\bigg|e^{itx}-1\bigg|dF_{n, j}(x)+2\max\limits_{1\leq j\leq k_{n}}\int\limits_{|x|\geq \epsilon}dF_{n, j}(x).
\end{split}
\]
Taking account of the inequality $|e^{iz}-1|\leq |z|$ valid for $z\in (-\infty, +\infty),$ it follows that
\[
\max\limits_{1\leq j\leq k_{n}}|f_{j}(t)-1| \leq \epsilon |t|+2\max\limits_{1\leq j\leq k_{n}}P(|X_{n, j}|\geq \epsilon).
\]
According to \cite{Gut2005}, we can deduced that
 \[
\mathbb{E}\bigg(\max\limits_{1\leq j\leq \nu_{n}}|f_{j}(t)-1|\bigg) \leq \epsilon |t|+2\mathbb{E}\bigg(\max\limits_{1\leq j\leq \nu_{n}}P(|X_{n, j}|\geq \epsilon)\bigg).
\]
Applying to (A), since $\epsilon>0$ is chosen arbitrarily small, the last inequality finishes the proof.
\end{proof}
\begin{remark}\label{rem:4.2}\quad If $\nu_{n}\stackrel{a.s}{=}k_{n}$ for $n\geq 1,$ then the above statements in  Proposition (\ref{pro:4.4}) deduce the results in \cite{Gnedenko1949} (Theorem 1, Pages 95--96). 
\end{remark}
\begin{proposition}\label{pro:4.5}\quad The random condition of asymptotic infinitesimality (RI) is deduced by the random Feller condition (RF), i.e., 
\[
(RF)\Longrightarrow (RI).
\]
\end{proposition}
\begin{proof}
By Chebyshev's inequality, for any $\epsilon>0,$ we have
\[
\mathbb{E}\bigg\{\max\limits_{1\leq j\leq \nu_{n}}P(|X_{n, j}|\geq \epsilon)\bigg\}\leq \frac{1}{\epsilon^{2}}\mathbb{E}\bigg\{\max\limits_{1\leq j\leq \nu_{n}}\sigma^{2}_{n, j}\bigg\}.
\]
Applying the (RF), we obtain the complete proof.
\end{proof}
Moreover, there is a condition stronger than (RL), named the random Lyapunov condition, which is presented in the following definition:
 \begin{definition}\label{def:4.5}\quad  The double sequence  $\{X_{n, j}\}$ is said to satisfy the random Lyapunov condition if
\begin{equation*}
\mathbb{E}\bigg\{\sum\limits_{j=1}^{\nu_{n}}\mathbb{E}\bigg(|X_{n, j}|^{2+\delta}\bigg)\bigg\}=o(1) \quad\text{as}\quad n\to\infty, \tag{R$\Lambda$}
\end{equation*}
for $\delta\in (0, 1].$  
\end{definition}
\begin{proposition}\label{pro:4.6}\quad The random condition Lyapunov (R$\Lambda$) follows the random Lindeberg condition (RL), that is, 
 \[
(R\Lambda)\Longrightarrow (RL).
\]
\end{proposition}
\begin{proof}
For  $\epsilon >0$ and $\delta\in (0,1],$ since $|X_{n, j}|\geq \epsilon$ implies $|X_{n, j}|^{\delta}\geq {\epsilon}^{\delta}.$ \\
Therefore, it is easily follows that
\[
\mathbb{E}\bigg\{\sum\limits_{j=1}^{\nu_{n}}\mathbb{E}\bigg(X^{2}_{n, j}\mathbb{I}(|X_{n, j}|\geq \epsilon) \bigg)\bigg\}\leq \frac{1}{\epsilon^{\delta}}\mathbb{E}\bigg\{\sum\limits_{j=1}^{\nu_{n}}\mathbb{E}\bigg(X^{2+\delta}_{n, j}\mathbb{I}(|X_{n, j}|\geq \epsilon) \bigg)\bigg\}\leq \frac{1}{\epsilon^{\delta}}\mathbb{E}\bigg\{\sum\limits_{j=1}^{\nu_{n}}\mathbb{E}\bigg(X^{2+\delta}_{n, j} \bigg)\bigg\}.
\]
This shows that if  the random Lyapunov condition (R$\Lambda$) is satisfied, the random Lindeberg condition (RL) is also valid.
\end{proof}
A typical theorem for the non-classical situation is the random Rotar theorem with the random Rotar condition, defined below. 
 \begin{definition}\label{def:4.6}\quad  The double sequence  $\{X_{n, j}\}$ is said to satisfy the random 
 Rotar condition if for any $\epsilon >0$
\begin{equation*}
\mathbb{E}\bigg\{\sum\limits_{j=1}^{\nu_{n}}\int\limits_{|x|>\epsilon}|x||F_{n, j}(x)-\Phi_{n, j}(x)|dx\bigg\}=o(1)\quad\text{as}\quad n\to\infty.\tag{RR}
\end{equation*}
\end{definition}
The next statement  shows that the  random Rotar condition (RR) is deduced by the random Lindeberg condition (RL). 
\begin{proposition}\label{pro:4.7}\quad The random condition Lindeberg (RL) follows the random Rotar condition (RR), that is, 
 \[
(RL)\Longrightarrow (RR).
\]
\end{proposition}
\begin{proof}
According to inequality (\ref{equ:2.3}), for some $C=C(\epsilon)>0$ and for any $\epsilon>0.$ we have
\begin{equation}\label{equ:4.4}
\mathbb{E}\bigg\{\sum\limits_{j=1}^{\nu_{n}}\int\limits_{|x|>\epsilon}|x||F_{n, j}(x)-\Phi_{n, j}(x)|dx\bigg\}
\leq C\bigg\{\mathbb{E}\bigg(\sum\limits_{j=1}^{\nu_{n}}\int\limits_{|x|> \epsilon}x^{2}dF_{n, j}(x)\bigg)+\mathbb{E}\bigg(\sup\limits_{1\leq j \leq \nu_{n}}\sigma^{2}_{n, j}\bigg)\bigg\}.
\end{equation}
Applying the random condition Lindeberg (RL), the first term on the right side of inequality (\ref{equ:4.4}) will vanish as $n\to\infty.$ This shows that the second term also tends to zero when $n\to\infty$ since the random Feller condition (RF) is followed by the random Lindeberg condition (RL). The proof is complete.
\end{proof}

\section{Central limit theorems for  random sums}\label{sec:5}
In this section, some typical central limit theorems for random sums of independent random variables of a double sequence like Lyapunov's, Lindeberg's, Feller's, and Rotar's theorems will be investigated.
\begin{theorem}\label{the:5.1}(Random Lindeberg theorem)\quad Suppose that the sequence $\{X_{n, j}\}$ is satisfied the random Lindeberg condition (RL), and assume that the random Feller condition holds (RF) for  the sequence $\{X^{*}_{n, j}\}$ of independent (in each row) normally distributed random variables which has zero means and finite variances $\sigma^{2}_{n, j}\in (0, \infty).$ Moreover, assume that the sequence $(\nu_{n}, n\geq 1)$ of  integer-valued, positive random variables, independent of all $X_{n, j}$ and $X^{*}_{n, j}$ for $1\leq j\leq k_{n}$ and $n\geq1.$   Furthermore, suppose that $\nu_{n}\stackrel{P}{\longrightarrow}\infty$ when $n\to\infty.$ Then, the central limit theorem for random sum (RCLT) is valid for the double sequence $\{X_{n, j}\},$ that is,  
\begin{equation*}
{\Delta_{n, \nu_{n}}}=\sup\limits_{x\in\mathbb{R}}\bigg|P\bigg(\sum\limits_{j=1}^{\nu_{n}}X_{n, j}<x \bigg)-\Phi_{0,1}(x) \bigg|=o(1), \quad\text{as}\quad n\to\infty.
\end{equation*}
\end{theorem}
\begin{proof}
According to the Lemma \ref{lem:3.1} and Proposition \ref{pro:3.2}, the sufficient condition for the validity of the random central limit theorem (RCLT) is
\[
\zeta_{3}\bigg(\sum\limits_{j=1}^{\nu_{n}}X_{n, j}, \sum\limits_{j=1}^{\nu_{n}}X^{*}_{n, j}\bigg)=o(1) \quad\text{as}\quad n\to\infty.
\]
Let us first consider the Zolotarev distance of random variables $X_{n, j}, X^{*}_{n, j}\in\mathfrak{X}$
\[
\zeta_{3}(X_{n, j}, X^{*}_{n, j})=\sup\limits_{f\in\mathcal{D}_{3}}\bigg|\mathbb{E}f(X_{n, j})-\mathbb{E}f(X^{*}_{n, j}) \bigg|,
\]
where 
\[
\mathcal{D}_{3}=\bigg\{f\in C^{2}(\mathbb{R}): |f^{\prime\prime}(x)-f^{\prime\prime}(y)|
\leq |x-y| \bigg\}.
\]
By the  Taylor series for $f\in C^{2}(\mathbb{R}),$ with assumptions that $\mathbb{E} X_{n, j}=0$ and $\mathbb{E}X^{2}_{n, j}=\sigma^{2}_{n, j}\in (0, \infty),$ we conclude that 
\begin{equation}\label{equ:5.1}
\mathbb{E}f(X_{n, j})=f(0)+\frac{f^{\prime\prime}(0)}{2}\sigma^{2}_{n, j}+\frac{1}{2}\mathbb{E}\bigg(\bigg[f^{\prime\prime}(\eta_{1})-f^{\prime\prime}(0) \bigg]X^{2}_{n, j}\bigg),
\end{equation}
where $\eta_{1}=\theta_{1}x, \theta_{1}\in (0, 1).$ Since $f\in \mathcal{D}_{3},$ for any $\delta_{1}>0$ there exists $\epsilon_{1}>0$ such that $\bigg|f^{\prime\prime}(\eta_{1})-f^{\prime\prime}(0) \bigg| <\delta_{1}$ when $|x|<\epsilon_{1}$ and $\bigg|f^{\prime\prime}(\eta_{1})-f^{\prime\prime}(0) \bigg| \leq 2||f^{(2)}||$ when $|x|\geq \epsilon_{1}.$\\
By a similar argument, with assumption that $X^{*}_{n, j}\stackrel{D}{=}\sigma_{n, j}X^{*},$ we have
\begin{equation}\label{equ:5.2}
\mathbb{E}f(X^{*}_{n, j})=f(0)+\frac{f^{\prime\prime}(0)}{2}\sigma^{2}_{n, j}+\frac{1}{2}\mathbb{E}\bigg(\bigg[f^{\prime\prime}(\eta_{2})-f^{\prime\prime}(0) \bigg](\sigma_{n, j}X^{*})^{2}\bigg),
\end{equation}
where $\eta_{2}=\theta_{2}x, \theta_{2}\in (0, 1).$ Since $f\in \mathcal{D}_{3},$ for any $\delta_{2}>0$ there exists $\epsilon_{2}>0$ such that $\bigg|f^{\prime\prime}(\eta_{2})-f^{\prime\prime}(0) \bigg| <\delta_{2}$ when $|x|<\epsilon_{2}/\sigma_{n, j}$ and  and $\bigg|f^{\prime\prime}(\eta_{2})-f^{\prime\prime}(0) \bigg| \leq 2||f^{(2)}||$ when $|x|\geq \epsilon_{2}/\sigma_{n, j}.$\\
Combining (\ref{equ:5.1}) and (\ref{equ:5.2}),  we conclude that
\begin{equation}\label{equ:5.3}
\begin{split}
&\bigg|\mathbb{E}f(X_{n, j})-\mathbb{E}f(X^{*}_{n, j})\bigg| \leq \frac{1}{2}\mathbb{E}\bigg(\bigg|f^{\prime\prime}(\eta_{1})-f^{\prime\prime}(0) \bigg| X^{2}_{n, j}\bigg)+\frac{1}{2}\mathbb{E}\bigg(\bigg|f^{\prime\prime}(\eta_{2})-f^{\prime\prime}(0) \bigg|(\sigma_{n, j}X^{*})^{2}\bigg)\\
& \leq \frac{1}{2}\bigg|f^{\prime\prime}(\eta_{1})-f^{\prime\prime}(0)\bigg|\mathbb{E}X^{2}_{n, j}+\frac{1}{2}\sigma^{2}_{n, j}\bigg|f^{\prime\prime}(\eta_{2})-f^{\prime\prime}(0)\bigg|\mathbb{E}( X^{*})^{2}\\
&\leq \frac{1}{2} \delta_{1}\sigma^{2}_{n, j} +||f^{(2)}||\int\limits_{|x|>\epsilon_{1}}x^{2}dF_{n, j}(x)
+\frac{1}{2} \delta_{2}\sigma^{2}_{n, j}+||f^{(2)}||\sigma^{2}_{n, j}\int\limits_{|x|>\epsilon_{2}/\sigma_{n, j}}x^{2}d\Phi_{0,1}(x)\\
&= \frac{1}{2} \delta \sigma^{2}_{n, j}+||f^{(2)}||\int\limits_{|x|>\epsilon_{1}}x^{2}dF_{n, j}(x)
+||f^{(2)}||\sigma^{2}_{n, j}\int\limits_{|x|>\epsilon_{2}/\sigma_{n, j}} x^{2}d\Phi_{0,1}(x) ,
\end{split}
\end{equation}
where $\delta=\frac{\delta_{1}+\delta_{2}}{2}>0,\epsilon_{1}>0, \epsilon_{2}>0.$\\
Summing over j  both sides of the (\ref{equ:5.3}), using the suggestion $\sum\limits_{j=1}^{k_{n}}\sigma^{2}_{n, j}=1,$ it follows that
\begin{equation}\label{equ:5.4}
\sum\limits_{j=1}^{k_{n}}\bigg|\mathbb{E}f(X_{n, j})-\mathbb{E}f(X^{*}_{n, j})\bigg| \leq 
 \frac{1}{2} \delta +||f^{(2)}||\sum\limits_{j=1}^{k_{n}}\int\limits_{|x|>\epsilon_{1}}x^{2}dF_{n, j}(x)
+||f^{(2)}||\int\limits_{|x|>\epsilon_{2}/\sigma_{n, j}} x^{2}d\Phi_{0,1}(x).
\end{equation}
Taking expectations  and multiplying $\sum\limits_{j=k_{n}}^{\infty}P(\nu_{n}=k_{n})$ both sides of (\ref{equ:5.4}), in view of Proposition \ref{pro:3.3}, we have
\begin{equation}\label{equ:5.5}
\begin{split}
\zeta_{3}\bigg(\sum\limits_{j=1}^{\nu_{n}}X_{n, j}, X^{*}\bigg)=\zeta_{3}\bigg(\sum\limits_{j=1}^{\nu_{n}}X_{n, j}, \sum\limits_{j=1}^{\nu_{n}}X^{*}_{n, j}\bigg)&\leq\frac{1}{2}\delta +||f^{(2)}||\mathbb{E}\bigg\{\sum\limits_{j=1}^{\nu_{n}}\int\limits_{|x|>\epsilon_{1}}x^{2}dF_{n, j}(x)\bigg\}\\
 &+||f^{(2)}||\mathbb{E}\bigg\{\int\limits_{|x|>\epsilon_{2}/\sigma^{*}}x^{2}d\Phi_{0,1}(x)  \bigg\},
\end{split}
\end{equation}
where $\epsilon_{1}>0, \epsilon_{1}>0, \delta>0$ and $\sigma^{*}=\max\limits_{1\leq j\leq \nu_{n}}\sigma_{n, j}.$ \\
Since $\delta>0$ is an arbitrary small, $||f^{(2)}||$ is bounded, the second term in the right side of inequality (\ref{equ:5.5})  vanishes when $n\longrightarrow\infty$ since the random Lindeberg condition is satisfied for the double sequence $\{X_{n, j}\}.$ The last term tends to zero under the random Feller condition is valid for the sequence $\{X^{*}_{n, j}\}$ in view of the Proposition \ref{pro:4.3}. The proof is finished.
\end{proof}
Next statement is  the random Lyapunov central limit theorem, whose proof is deduced from the Theorem \ref{the:5.1}.
\begin{theorem}\label{the:5.2}(Random Lyapunov's theorem) \quad Suppose that, for $n\geq 1,$ the sequence $\{X_{n, j}\}$ is satisfied the random Lyapunov condition (R$\Lambda$), and assume that for  the sequence $\{X^{*}_{n, j}\}$ of independent normal distributed random variables which has zero means and finite variances $\sigma^{2}_{n, j}\in (0, \infty),$  the random Feller condition holds. Moreover, we assume that the sequence  $(\nu_{n}, n\geq 1)$ of integer-valued, positive random variables, independent of all $X_{n, j}$ and $X^{*}_{n, j}$ for $1\leq j\leq n$ and $n\geq1.$  Furthermore, suppose that $\nu_{n}\stackrel{P}{\longrightarrow}\infty$ when $n\to\infty.$ Then, the central limit theorem for random sum (RCLT) is valid for the double sequence $\{X_{n, j}\},$ i.e.,  
\begin{equation*}
{\Delta_{n, \nu_{n}}}=\sup\limits_{x\in\mathbb{R}}\bigg|P\bigg(\sum\limits_{j=1}^{\nu_{n}}X_{n, j} <x\bigg)-\Phi_{0,1}(x) \bigg|=o(1) \quad\text{as}\quad n\to\infty.
\end{equation*}
\end{theorem}
\begin{proof}
The proof is deduced by the inequality (\ref{equ:4.1}), that is, since the following implications:
\[
(R\Lambda)\Longrightarrow (RL) \Longrightarrow (RCLT), \quad\text{as}\quad n\to\infty,
\]
are valid, for any $\epsilon>0$ and $\delta\in (0, 1].$
\end{proof}
Based on the fact that the random Lindeberg condition is automatically satisfied for the double sequence of i.i.d. (in each row) random variables, the following statement is a direct corollary of the Theorem \ref{the:5.1}.
\begin{corollary}\label{cor:5.1}\quad The random central limit theorem (RCLT) holds for the sequence of independent identically distributed random variables.
\end{corollary}
\begin{theorem}\label{the:5.3}(Random Feller theorem)\quad The random Lindeberg condition is necessary condition for the random central limit theorem under the random Feller condition, that is,
\begin{equation*}
(RCLT)\quad\&\quad (RF) \Longrightarrow (RL)\quad\text{as}\quad n\to\infty.
\end{equation*}
\end{theorem}
\begin{proof}
Following  \cite{Chen2005} (Theorem 5.3, Formula (5.17), Page 36), for any $\epsilon>0,$ we have
\begin{equation}\label{equ:5.6}
\sum\limits_{j=1}^{k_{n}}\int\limits_{|x|\geq \epsilon}x^{2}dF_{n, j}(x)\leq \frac{C}{(1-e^{-\epsilon^{2}/4})}\bigg( \Delta_{n, k_{n}}+\sum\limits_{j=1}^{k_{n}}\sigma^{4}_{n, j}\bigg),
\end{equation}
where $C$ is an absolute constant. \\
With assumption that $\sum\limits_{j=1}^{k_{n}}\sigma^{2}_{n, j}=1,$ we obtain
\begin{equation}\label{equ:5.7}
\sum\limits_{j=1}^{k_{n}}\sigma^{4}_{n, j}\leq \max\limits_{1\leq j\leq k_{n}}\sigma^{2}_{n, j}\times \bigg(\sum\limits_{j=1}^{k_{n}}\sigma^{2}_{n, j}\bigg)=\max\limits_{1\leq j\leq k_{n}}\sigma^{2}_{n, j}.
\end{equation}
Combining (\ref{equ:5.7}) with (\ref{equ:5.6}), it follows that
\begin{equation}\label{equ:5.8}
\sum\limits_{j=1}^{k_{n}}\int\limits_{|x|\geq \epsilon}x^{2}dF_{n, j}(x)\leq \frac{C}{(1-e^{-\epsilon^{2}/4})}\bigg( \Delta_{n, k_{n}}+\max\limits_{1\leq j\leq k_{n}}\sigma^{2}_{n, j}\bigg),
\end{equation}
for any  $\epsilon>0,$ and $C$ is an absolute  positive constant. \\
From (\ref{equ:5.8}), we conclude that  
\begin{equation}\label{equ:5.9}
\mathbb{E}\bigg\{\sum\limits_{j=1}^{\nu_{n}}\int\limits_{|x|\geq \epsilon}x^{2}dF_{n, j}(x)\bigg\}\leq \frac{C}{(1-e^{-\epsilon^{2}/4})} \bigg(\Delta_{n, \nu_{n}}+\mathbb{E}\bigg[\max\limits_{1\leq j\leq \nu_{n}}\sigma^{2}_{n, j}\bigg]\bigg).
\end{equation}
for any  $\epsilon>0,$ and $C$ is an absolute positive constant. \\
 It is clear that if the random central limit theorem (RCLT) is satisfied under the random Feller condition (RF), the random Lindeberg condition (RL) is valid. The proof of the theorem is complete.
 \end{proof}
The preceding conditions, such as (R$\Lambda$), (RL), (RF), and (RI), play an important role in the validity of the central limit theorems in classical situations. Following Zolotarev \cite{Zolotarev1997}, the limit theorems that make no use of the condition (RI) are said to be non-classical. The next result is a random central limit theorem in non-classical situations.
\begin{theorem}\label{the:5.4} (Random Rotar's theorem)\quad  For every $\epsilon>0,$ the following statements are true:
\begin{enumerate}
\item[(\bf {A}).] The central limit theorem for random sum of independent random variables (RCLT) is deduced by  the random Rotar condition (RR), i.e., 
\[
(RR)\Longrightarrow (RCLT) \quad\text{as}\quad n\to\infty.
\]
\item[(\bf {B}).] Under the random Feller condition (RF), the random central limit theorem of independent random variables (RCLT) deduces the random Rotar condition (RR), that is, 
\[
(RF)\quad\&\quad (RCLT)  \Longrightarrow (RR) \quad\text{as}\quad n\to\infty.
\]
\end{enumerate}
 \end{theorem}
\begin{proof}
 ({\bf A}).\quad From \cite{Shiryaev1996} (Theorem 1, formula (6), page 339), it follows that
  \begin{equation}\label{equ:5.10}
\bigg|\mathbb{E}(e^{itS_{n, \nu_{n}}})-e^{-t^{2}/2}\bigg|\leq \epsilon |t|^{3}+2t^{2}\mathbb{E}\bigg\{\sum\limits_{j=1}^{\nu_{n}}\int\limits_{|x|>\epsilon}|x||F_{n, j}(x)-\Phi_{n, j}(x)|dx\bigg\},
\end{equation}
here $\epsilon>0.$ Since $\epsilon >0$ is an arbitrary small and the random Rotar condition  (RR) is satisfied, from above inequality  (\ref{equ:5.10}), it may be concluded that
\[
\bigg|\mathbb{E}(e^{itS_{\nu_{n}}})-e^{-t^{2}/2}\bigg|=o(1)\quad\text{as}\quad n\to\infty.
\] 
In view of the continuity theorem of characteristic functions, the above  limit expression deduces  to the random central limit theorem for a double sequence $\{X_{n, j}\}$ is valid, that is
\[
\Delta_{n, \nu_{n}}=\sup\limits_{x\in\mathbb{R}}\bigg|P\bigg(\sum\limits_{j=1}^{\nu_{n}}X_{n, j}<x \bigg)-\Phi_{0, 1}(x) \bigg|=o(1)\quad\text{as}\quad n\to\infty.
\]
It confirms that the first assertion ({\bf A}) is valid.\\
({\bf B}).\quad On account of $\sum\limits_{j=1}^{k_{n}}\sigma^{2}_{n, j}=1,$ it is clear that 
\[
\sum\limits_{j=1}^{k_{n}}\sigma^{4}_{n, j}\leq \max\limits_{1\leq j\leq k_{n}}\sigma^{2}_{n, j}. 
\]
 According to \cite{Formanov2019} (Theorem 1,  Part 3, Page 33),  we have
\[
\sum\limits_{j=1}^{k_{n}}\int\limits_{|x|>\epsilon}|x||F_{n, j}(x)-\Phi_{n, j}(x)|dx 
\leq C(\epsilon)\bigg(\Delta_{n, k_{n}}+\sum\limits_{j=1}^{k_{n}}\sigma^{4}_{n, j} \bigg)\leq C(\epsilon)\bigg(\Delta_{n, k_{n}}+ \max\limits_{1\leq j\leq k_{n}}\sigma^{2}_{n, j}\bigg),
\]
where $C(\epsilon)>0,$ for any $\epsilon>0.$ \\
Therefore, from the above inequality, it may be conclude that
 \begin{equation}\label{equ:5.11}
\mathbb{E}\bigg\{\sum\limits_{j=1}^{\nu_{n}}\int\limits_{|x|>\epsilon}|x||F_{n, j}(x)-\Phi_{n, j}(x)|dx \bigg\}\leq C(\epsilon)\bigg[\Delta_{n, \nu_{n}}+\mathbb{E}\bigg(\max\limits_{1\leq j\leq \nu_{n}}\sigma^{2}_{j}\bigg) \bigg],
\end{equation}
where $C(\epsilon)>0,$ for any $\epsilon>0.$ In view of (\ref{equ:5.11}), the random Rota condition is satisfied if the random central limit theorem holds (i.e., $\Delta_{n, \nu_{n}}=o(1)$  as $n\to\infty$) under the random Feller condition (RF). Hence, the second  statement ({\bf B}) 
\[
 (RF)\quad\&\quad (RCLT)\Longrightarrow (RR)
\]
is valid. This completes the proof of the theorem.
\end{proof}

\section{Concluding  remarks}\label{sec:5}
To end this paper, we consider a case when the double sequence of independent random variables (in each row) becomes a single sequence, which has many advantages when studying sample statistics, estimating, statistical hypothesis testing, etc. This is done with the linear transformation "series form", much discussed in \cite{Chen2005}, \cite{Gnedenko1949}, \cite{Petrov1995},  \cite{Rotar1996},  \cite{Shiryaev1996}, and \cite{Zolotarev1997}.

Let us consider a simple sequence $(X_{j}, j\geq 1)$ of independent (not necessarily identically distributed) random variables, defined on a probability space $(\Omega, \mathcal{A}, \mathbb{P}).$ From now on, assume that the sequence   $(X_{j}, j\geq 1)$ has expectations $\mathbb{E} X_{j}=a_{j}$ and finite variance $\mathbb{D}X_{j}=\sigma^{2}_{j}\in (0,+\infty),$ for $j\geq 1.$ Denote $B^{2}_{n}=\sum\limits_{j=1}^{n}\sigma^{2}_{j},$ and suggest that $B_{n}=\sqrt{\sum\limits_{j=1}^{n}\sigma^{2}_{j}}>0$ for $n\geq 1$ and $B_{n}\to\infty$ as $n\to\infty.$ 

 From the obtained results in previous parts, the analogous ones may be valid for the sequence $(X_{j}, j\geq 1)$ if the "series form"  $X_{n, j}=\frac{X_{j}-a_{j}}{B_{n}}$ for $j\geq 1, n\geq 1,$ will be used.  
 
 Now,  by setting $S_{n}:=\frac{1}{B_{n}}\sum\limits_{j=1}^{n}(X_{j}-a_{j}),$ we obtain $\mathbb{E}S_{n}=0$ and $\mathbb{D}S_{n}=1.$ 
 
 Thus, our research subject is now the random central limit theorem for the sequence $(X_{j}, j\geq 1),$ presented as follows:
 \begin{equation*}
 \overline{\Delta}_{\nu_{n}}:=\sup\limits_{x\in\mathbb{R}}\bigg|P\bigg(\frac{1}{B_{\nu_{n}}}\sum\limits_{j=1}^{\nu_{n}}(X_{j}-a_{j})<x \bigg)-\Phi_{0,1}(x) \bigg|=o(1)\quad\text{as}\quad n\to\infty,\tag{$\overline{RCLT}$}
 \end{equation*}
 where $(n_{n}, n\geq 1)$ denotes the sequence of integer-valued, positive random variables, independent of all $X_{j}$ for $j\geq 1.$ Assume that $\nu_{n}\stackrel{P}{\longrightarrow}\infty$ and $B^{2}_{\nu_{n}}:=\sum\limits_{j=1}^{\nu_{n}}\sigma^{2}_{j}\longrightarrow\infty$ when $n\to\infty.$
  
Specifically, the formulas analogous to those in Sections 3 and 4 are presented as follows:
\begin{enumerate}
\item A sequence $(X_{j}, j\geq 1)$  is said to satisfy random Lindeberg condition  if for any $\epsilon >0$ 
\begin{equation*}
\mathbb{E}\bigg\{\frac{1}{B^{2}_{\nu_{n}}}\sum\limits_{j=1}^{\nu_{n}}\mathbb{E}\bigg[(X_{j}-a_{j})^{2}\mathbb{I}(|X_{j}-a_{j}|>\epsilon B_{\nu_{n}})\bigg]\bigg\}=o(1)\quad\text{as}\quad n\to\infty.\tag{$\overline{RL}$}
\end{equation*} 
\item The random Liyapunov condition is valid for a sequence $(X_{j}, j\geq 1)$  if for any $\delta \in (0, 1]$ 
\begin{equation*}
\mathbb{E}\bigg\{\frac{1}{B^{2+\delta}_{\nu_{n}}}\sum\limits_{j=1}^{\nu_{n}}\mathbb{E}\bigg(|X_{j}-a_{j}|^{2+\delta}\bigg) \bigg\}=o(1)\quad\text{as}\quad n\to\infty.\tag{$\overline{R\Lambda}$}
\end{equation*} 
\item A sequence  $(X_{j}, j\geq 1)$  is said to obey the random Feller condition if 
\begin{equation*}
\mathbb{E}\bigg\{\frac{1}{B^{2}_{\nu_{n}}}\max\limits_{1\leq j\leq \nu_{n}}\mathbb{E}(X_{j}-a_{j})^{2}\bigg\}=o(1)\quad\text{as}\quad n\to\infty.\tag{$\overline{RF}$}
\end{equation*} 
\item The independent individual summands $X_{j} (j\geq 1)$ in random sum $\frac{1}{B_{\nu_{n}}}\sum\limits_{j=1}^{\nu_{n}}(X_{j}-a_{j})$  are said to satisfy the random condition of asymptotic infinitesimality if, for any $\epsilon >0$ 
\begin{equation*}
\mathbb{E}\bigg[\max\limits_{1\leq j\leq \nu_{n}}P\big(|X_{j}-a_{j}|\geq \epsilon B_{\nu_{n}}\big) \bigg]=o(1)\quad\text{as}\quad n\to\infty.\tag{$\overline{RI}$}
\end{equation*} 
\item The random Rotar condition is defined in the following form:
\begin{equation*}
\mathbb{E}\bigg\{\frac{1}{B_{\nu_{n}}}\sum\limits_{j=1}^{\nu_{n}}\int\limits_{|x|>\epsilon B_{\nu_{n}}}|x-a_{j}||F_{j}(x)-\Phi_{j}(x)|dx\bigg\}=o(1)\quad\text{as}\quad n\to\infty,\tag{$\overline{RR}$}
\end{equation*}
for any $\epsilon >0.$ Here $F_{j}(x)=P(X_{j}-a_{j}<x)$ and $\Phi_{j}(x)=\Phi_{0,1}(\frac{x-a_{j}}{\sigma_{j}}).$
\end{enumerate}
Now we have the relationships between the above conditions and the  central limit theorems in classical and non-classical situations, which present the following  implications:
\begin{enumerate}
\item[({\bf A.})]\quad 
\[
(\overline{R\Lambda})\Longrightarrow (\overline{RL})\Longrightarrow (\overline{RF}) \Longrightarrow (\overline{RI}) \quad\text{as}\quad n\to\infty. 
\]
\item[({\bf B.})]\quad 
\[
(\overline{RL})\Longrightarrow (\overline{RR})\quad\text{as}\quad n\to\infty. 
\]
\item[({\bf D.})]\quad The random Lyapunov central limit theorem states that
\[
(\overline{R\Lambda})\Longrightarrow (\overline{RCLT})\quad\text{as}\quad n\to\infty. 
\]
\item[({\bf E.})]\quad The random Lindeberg central limit theorem confirms that
\[
(\overline{RL})\Longrightarrow (\overline{RCLT})\quad\text{as}\quad n\to\infty. 
\]
\item[({\bf F.})]\quad The random Feller central limit theorem formulates that
\[
(\overline{RF})\quad\&\quad  (\overline{RCLT})\Longrightarrow (\overline{RL})\quad\text{as}\quad n\to\infty. 
\]
\item[({\bf G.})]\quad The random Rota central limit theorem says that
\begin{enumerate}
\item[{i.}] 
\[
(\overline{RR})\Longrightarrow (\overline{RCLT}), \quad\text{as}\quad n\to\infty. 
\]
\item[{ii.}] 
\[
(\overline{RCLT})\quad\&\quad  (\overline{RF})\Longrightarrow (\overline{RR})\quad\text{as}\quad n\to\infty. 
\]
\end{enumerate}
\end{enumerate}

\noindent
{\bf Acknowledgment.}\quad    The author wishes to express his thanks to Dr. Nguyen Tran Thuan from the Department of Mathematics at Saarland University (Germany) for his help in getting the necessary books for writing this article.

\noindent
Author's address:\\
Tran Loc Hung\\
Ho Chi Minh City, Vietnam.\\
tlhungvn@gmail.com

\end{document}